\documentclass[natbib]{article}       
\usepackage{graphicx}
\usepackage{geometry}
\usepackage{subfigure}
\usepackage{amsthm}
\usepackage{amsmath}
\usepackage{bm}
\usepackage{amssymb}
\usepackage{vector}
\usepackage{enumerate}
\usepackage{amsfonts}
\usepackage{bm}
\usepackage{multicol}
\usepackage{longtable}
\usepackage{array}
\usepackage{mathrsfs}
\usepackage{extarrows}
\usepackage{here}
\usepackage{color}
\usepackage{vector}
\usepackage{bbm}

\newtheorem{theorem}{Theorem}[section]
\newtheorem{lemma}[theorem]{Lemma}

\newtheorem{assumption}{Assumption}
\newtheorem{corollary}[theorem]{Corollary}
\newtheorem{remark}[theorem]{Remark}

\newcommand{\field}[1]{\mathbb{#1}} 
 
\newcommand{\R}{\field{R}}

\newcommand{\Z}{\field{Z}}

\newcommand{\ind}{\mathbbm{1}}
\newcommand{\ep}{\epsilon}

\newcommand{\var}{\mathrm{Var}}
\newcommand{\cov}{\mathrm{Cov}}
\newcommand{\cum}{\mathrm{cum}}

\newcommand{\freqint}{\int^{\pi}_{-\pi}}

\newcommand{\sumn}[1][i]{\sum_{#1 = 1}^n}

\providecommand{\abs}[1]{\lvert#1\rvert}
\providecommand{\Babs}[1]{\Bigl \lvert#1\Bigr \rvert} 
\providecommand{\norm}[1]{\lVert#1 \rVert}

\providecommand{\prt}[1]{\{\, #1 \,\}}
\providecommand{\lprt}[1]{\left\{\, #1 \,\right\}}
\providecommand{\fp}[2]{\frac{\partial^{#2}}{\partial #1^{#2}}}

\newcommand{\plim}{\xrightarrow{{\mathcal P}}}
\newcommand{\dlim}{\xrightarrow{{\mathcal L}}}

\begin{document}

\title{Quantile tests in frequency domain for sinusoid models}



\author{Yan Liu}
\date{}



\maketitle

\begin{abstract}
For second order stationary processes, 
the spectral distribution function is uniquely determined by the autocovariance functions of the processes.
We define the quantiles of the spectral distribution function and propose two estimators for the quantiles.
Asymptotic properties of both estimators are elucidated and 
the difference from the quantile estimators in time domain is also indicated.
We construct a testing procedure of quantile tests from the asymptotic distribution of the estimators
and strong statistical power is shown in our numerical studies.
\end{abstract}


\section{Introduction}
Nowadays, the quantile based estimation becomes a notable method in statistics.
Not only statistical inference for the quantile of cumulative distribution function is considered,
the quantile regression, a method taking place of the ordinary regression, is also broadly used
for statistical inference. (See \cite{koenker2005}.)
In the area of time series analysis, however, 
the quantile based inference is still undeveloped yet.
A fascinating approach in frequency domain, called ``quantile periodogram''
is proposed and studied in \cite{li2008, li2012}.
The method associated with copulas, quantiles and ranks are developed in  \cite{dette2015}.

As there exists a well-behaved spectral distribution function for second order stationary process, 
we introduce the quantile of the spectral distribution and 
develop a statistical inference theory for it.
We also propose a quantile test in frequency domain to test the dependence structure
of second order stationary process, 
since the spectral distribution function is uniquely determined by
the autocovariance functions of the process.

In the context of time series analysis, 
\cite{whittle1952b} mentioned that ``the search for periodicities'' constituted the whole 
of time series theory.
He proposed an estimation method based on a nonlinear model driven by a simple harmonic component.
After the work, to estimate the frequency has been a remarkable statistical analysis.
A sequential literature by
\cite{whittle1952b}, \cite{walker1971}, \cite{hannan1973}, \cite{rr1988} and \cite{qt1991}
investigated the method proposed by \cite{whittle1952b} and
pointed out the misunderstandings in \cite{whittle1952b}, respectively.
The noise structure is also generalized from independent and identically distributed 
white noise to the second order stationary process.
The main result in those works revealed the properties of the periodogram
and showed that the convergence factor of the estimator for the frequencies is $n^{3/2}$,
which is different from well known order $n^{1/2}$, although the asymptotic distribution
of the method is Gaussian.

\cite{qh2001} reviewed all the results above and proposed an alternative approach based on 
an iterative ARMA method.
In reality, they found that the nonlinear model for $\{y_t\}$ with 
a peculiar frequency structure plus stationary process $\{x_t\}$, 
called ``sinusoid models", such that
$y_t = A \cos(\lambda t + \phi) + x_t$ can be rewritten, by the trigonometric relation, as
$y_t - \beta y_{t-1} + y_{t -2} = x_t - \alpha x_{t-1} + x_{t-2}$,
where $\alpha = \beta$ depend on the peculiar frequency.
The method can be summarized by estimating $\beta$ for given $\alpha$ and
substituting $\beta$ for $\alpha$ until both $\alpha$ and $\beta$ converge.

Different from all the methods above,
we employ the check function to estimate quantiles,
the frequencies of spectral distribution function, for second order stationary process.
In view of correspondence between the spectral density function and the periodogram
for the stationary process,
we first directly apply the objective function to the bare periodogram.
It is expected the asymptotic normality of the approach
from the result by \cite{hosoya1989} on the bracketing condition in frequency case.
The approach for estimating in frequency domain certainly
has the consistency for the true value.
However, asymptotic normality of the quantile estimator based on the bare periodogram does not hold,
which is obviously different from the quantile estimation theory in time domain.
We give the results on the asymptotic properties of the estimator and modified the estimator. 
The modified estimator, by the method of smoothing, is asymptotically normal distributed.
We extended our result to the sinusoid models and applied the asymptotic distribution
to the quantile tests in the frequency domain.

The notations and symbols used in this paper are listed in the following:
for a vector or a matrix $A$, 
$A_j$ and $A_{ij}$, respectively, 
denote the $j$th and the $(i, j)$th element of corresponding vector and matrix;
$A'$ denotes the transpose of the matrix $A$;
$\cum(\irvec{X})$ denotes the joint cumulant of the random variables $\{\irvec{X}\}$;
for stationary process $\{X_t\}$, the joint cumulant $\cum_X(u_1, \dots, u_{n-1})$ 
simply denotes
$\cum(X_t, X_{t + u_1}, \dots, X_{t + u_{n - 1}})$; 
$L^p$ denotes the space of complex-valued functions on $[-\pi, \pi]$,
equipped with $L^p$ norm $\norm{g}_p$, i.e., $\{\freqint \abs{g(\omega)}^p d\omega\}^{1/p}$;
$\ind(\cdot)$ denotes the indicator function;
$e$ denotes the Napier's constant;
$I_d$ denotes the $d$-dimensional identity matrix;
$\plim$ and $\dlim$ denote the convergence in probability and the convergence in law, respectively.

\section{Preliminaries}
In this section, we review the spectral distribution functions of second order stationary processes
and introduce the quantiles of the spectral distribution functions.
\label{sec:2}
Suppose $\{X_t \,;\, t\in \Z\}$ is a zero mean second order stationary process with finite autocovariance function 
$R_X(h) = \cov(X_{t + h}, X_t)$, for $h \in \Z$. 
From Herglotz's theorem,
there exists a right continuous, non-decreasing, bounded distribution function $F_X(\omega)$
on $[-\pi, \pi]$ for the autocovariance function $R_X(h)$ of the process
such that
\begin{equation*}
 R_X(h) = \freqint e^{-i h \omega} F_X(d\omega), \quad (h \in \Z).
\end{equation*}
Explicitly, the spectral distribution function $F_X(\omega)$ is represented by
\begin{equation} \label{eq:2.1}
F_X(\omega)
=
\lim_{n \to \infty}
\frac{1}{2\pi} \sum_{h = - n}^{n}
R_X(h) \frac{\exp(- i \omega h) - 1}{-i h}.
\end{equation}

The structure of the second order stationary process can be discriminated by their own spectral distribution
function $F_X(\omega)$.
Below, we give 4 figures of spectral distribution functions of second order Gaussian stationary processes,
including White noise, MA(1) process with coefficient 0.9, AR(1) process with coefficient 0.9 and -0.9.

\begin{figure}[H]
 \subfigure[White noise]
 {\includegraphics[clip, width=0.4\columnwidth]{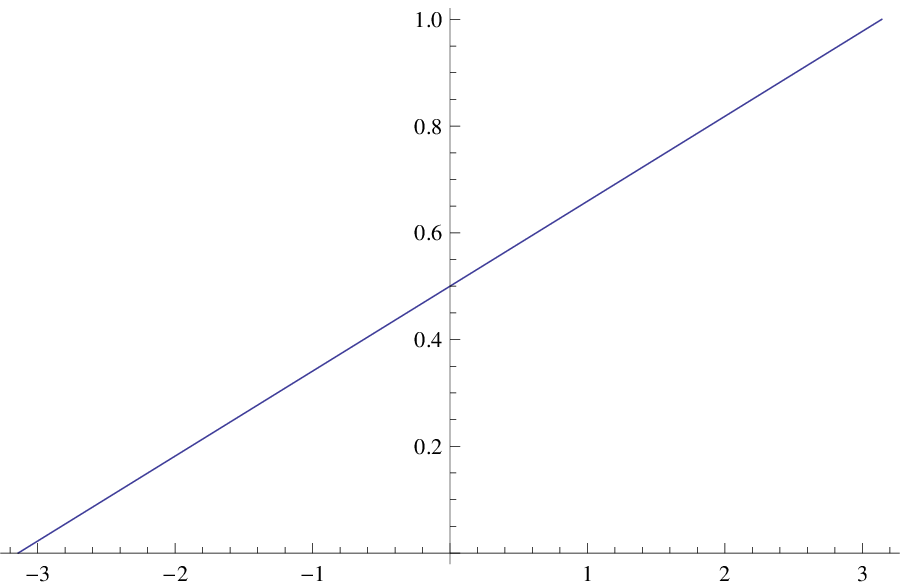}} \qquad 
 \subfigure[MA(1) process with coefficient 0.9]
 {\includegraphics[clip, width=0.4\columnwidth]{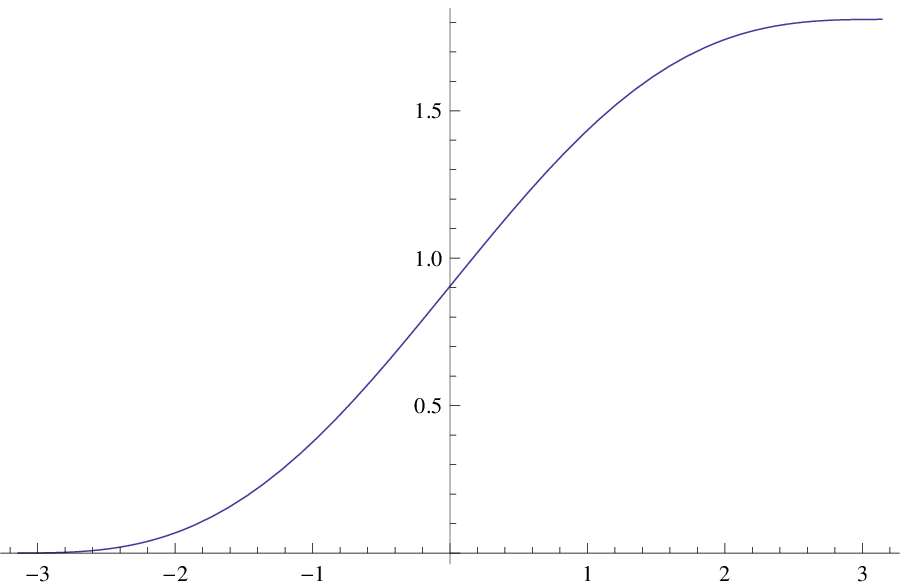}} 
\end{figure}

\begin{figure}[H]
 \subfigure[AR(1) process with coefficient 0.9]
 {\includegraphics[clip, width=0.4\columnwidth]{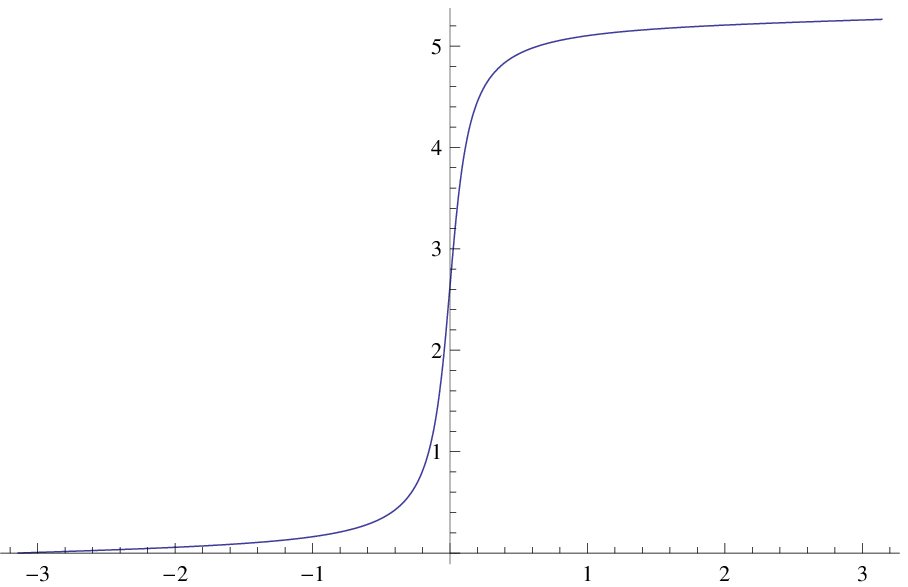}} \qquad 
 \subfigure[AR(1) process with coefficient -0.9]
 {\includegraphics[clip, width=0.4\columnwidth]{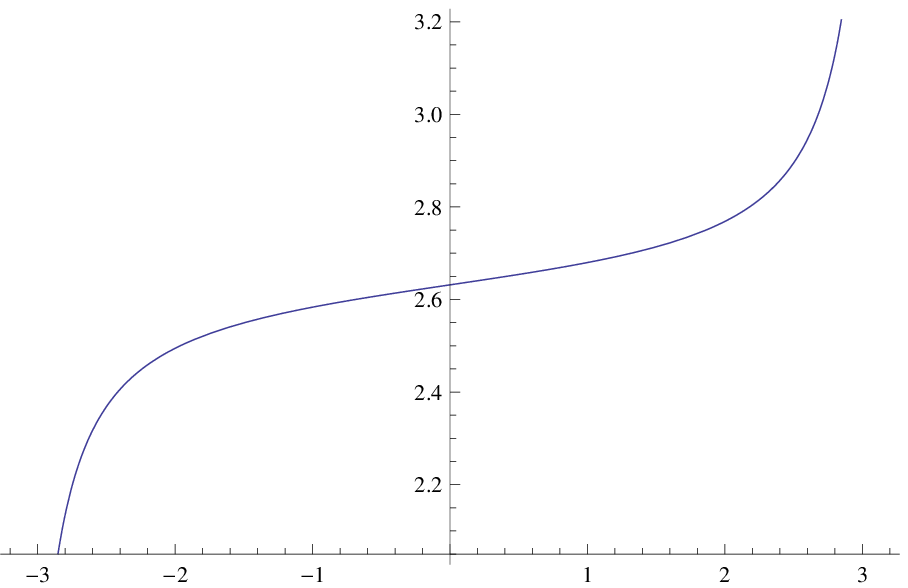}} 
 \caption{Figures of spectral distributions of second order stationary processes}
 \label{fig:1}
\end{figure}

To be specific, if the spectral distribution function $F_X(\omega)$ is absolutely continuous
with respect to Lebesgue measure, then $\{X_t\}$ has the spectral density $f_X(\omega)$,
which is corresponding to the $h$th autocovariance function $R_X(h)$ by
\[
R_X(h) = \freqint e^{-i h \omega} f_X(\omega) d\omega.
\]

Next, we introduce the $p$th quantile $\lambda_p$ of the spectral distribution function $F_X(\omega)$.
For simplicity, write $R_X(0) = \Sigma_X$.
Note that the spectral distribution function $F_X(\omega)$ takes value on $[0, \Sigma_X]$.
The generalized inverse distribution function $F^{-1}_X(\psi)$ for $0 \leq \psi \leq \Sigma_X$
is defined by
\begin{equation*} 
 F_X^{-1}(\psi) = \inf\{\omega\,; F_X(\omega) \geq \psi\}.
\end{equation*}
For $0 \leq p = \Sigma_X^{-1} \psi \leq 1$, we define the $p$th quantile $\lambda_p$ as
\begin{equation} \label{eq:2.2}
 \lambda_p := F_X^{-1}(\Sigma_X^{-1} \psi) = \inf\{\omega\,; F_X(\omega)\Sigma_X^{-1} \geq p\}.
\end{equation}

Define $\Lambda = [-\pi, \pi]$.
In the following, we show that the $p$th quantile $\lambda_p$ can be defined by
the minimizer of the following objective function $S(\theta)$, i.e.,
\begin{equation} \label{eq:2.3}
S(\theta)
=
\freqint \rho_p(\omega - \theta) F_X(d\omega),
\end{equation}
where $\rho_{\tau}(u)$, called ``the check function'' (e.g.\ \cite{koenker2005}), is defined as
\begin{equation*}
 \rho_{\tau}(u) = u(\tau - \ind(u < 0)).
\end{equation*}

\begin{theorem} \label{thm:2.1}
Suppose $\{X_t; \, t \in \Z \}$ is a zero mean
second order stationary process with spectral distribution function $F_X(\omega)$. 
Define $S(\theta)$ by \eqref{eq:2.3}.
Then the $p$th quantile $\lambda_p$ of the spectral distribution $F_X(\omega)$
is a minimizer of $S(\theta)$.
Furthermore, $\lambda_p$ is unique and satisfies
\begin{equation} \label{eq:2.4}
\lambda_p = \inf\{\omega \in \Lambda;\, S(\omega) = \min_{\theta \in \Lambda} S(\theta)\}.
\end{equation}
\end{theorem}

The representation \eqref{eq:2.4} of the $p$th quantile $\lambda_p$ of the spectral distribution function $F_X(\omega)$
is useful when we consider the estimation theory of $\lambda_p$.
From the definition of the spectral distribution function $F_X(\omega)$,
$F_X(\omega)$ is uniquely determined by the autocovariance function $R_X(h)$ $(h \in \Z)$.
Accordingly, the dependence structure of the second order stationary process $\{X_t;\, t \in \Z\}$
can be discriminated by the $p$th quantile $\lambda_p$ since $\lambda_p \not = \lambda_p'$
if $p \not = 0, 1/2, 1$ and $F_X(\omega) \not = c F_X'(\omega)$, $c \in \R$.

Let us consider the estimation procedure for $\lambda_p$.
Suppose the observation stretch of the process is defined by $\{X_t \,;\, 1 \leq t \leq n\}$.
The parameter space for the $p$th quantile $\lambda_p$ is defined by $\Lambda$. 
$\lambda_p$ is in the interior of $\Lambda$.
The objective function $S_n(\theta)$ for estimation can be defined by
\begin{equation} \label{eq:2.5}
 S_n(\theta) = \freqint \rho_{p}(\omega - \theta) I_{n, X}(\omega) d\omega,
\end{equation}
where $I_{n, X}(\omega)$ is the periodogram based on the observation stretch, and defined by
\begin{equation} \label{eq:2.6}
 I_{n, X}(\omega) = \frac{1}{2 \pi n} 
 \Babs{\sum_{j=1}^n X_j e^{ij \omega}}^2.
\end{equation} 
Hence, the estimator $\hat{\lambda}_p$ for $\lambda_p$ can be defined by
\begin{equation} \label{eq:2.7}
 \hat{\lambda}_p \equiv \hat{\lambda}_{p, n}  = \arg \min_{\theta \in \Lambda} S_n(\theta).
\end{equation}

\section{Asymptotic distribution of $\hat{\lambda}_p$ for stationary processes} \label{sec:3}
In this section, we consider the asymptotic properties of the estimator $\hat{\lambda}_p$ defined by \eqref{eq:2.7}
for stationary process $\{X_t; \, t \in \Z\}$ under the following assumptions.

\begin{assumption} \label{asp:3.1}
\begin{enumerate}[{\rm (i)}]
 \item $\{X_t\}$ is a zero mean,
 strictly stationary real valued process, all of whose moments exist with
 \[
 \sum_{u_1, \dots, u_{k-1} = -\infty}^{\infty} \abs{\cum_X(u_1, \dots, u_{k-1})} < \infty,
 \quad
 \text{for $k = 2, 3, \dots$}.
 \]
 \item $f_X(\omega) \in \text{Lip}(\alpha)$ for $\alpha > 1/2$.
\end{enumerate}
\end{assumption}
Under Assumption \ref{asp:3.1}, the fourth order spectral density is defined by
\[
Q_X(\omega_1, \omega_2, \omega_3)
=
\frac{1}{(2\pi)^3}
\sum_{t_1, t_2, t_3 = -\infty}^{\infty}
\exp\{-i (\omega_1 t_1 + \omega_2 t_2 + \omega_3 t_3)\}
\cum_X(t_1, t_2, t_3).
\]

First, we show the consistency of the estimator $\hat{\lambda}_p$ under Assumption \ref{asp:3.1}.
\begin{theorem} \label{thm:3.2}
Suppose $\{X_t;\, t\in \Z\}$ satisfies Assumption \ref{asp:3.1} and
the $p$th quantile $\lambda_p$ of the spectral distribution
of $\{X_t\}$ is defined by \eqref{eq:2.2}. 
If 
$\hat{\lambda}_p$ is defined by \eqref{eq:2.7}, then we have
\[
\hat{\lambda}_p \plim \lambda_p. 
\]
\end{theorem}

The consistency of the estimator \eqref{eq:2.7} is not difficult to expect.
The result, however, requires the continuity of the spectral distribution function $F_X(\omega)$,
a strong assumption, if we stand on the estimator \eqref{eq:2.7}.
We will modify the estimator \eqref{eq:2.7} by a new estimator later to loose Assumption \ref{asp:3.1}.

Next, we investigate the asymptotic distribution of the estimator $\hat{\lambda}_p$.
We impose the following assumption on $\{X_t\}$ instead of Assumption \ref{asp:3.1},
which is stronger than Assumption \ref{asp:3.1}.
\begin{assumption} \label{asp:3.2}
 $\{X_t\}$ is a zero mean, strictly stationary real valued process, all of whose moments exist with
 \[
 \sum_{u_1, \dots, u_{k-1} = -\infty}^{\infty} 
 \Bigl(
 1 + \sum_{j = 1}^{k - 1} \abs{u_j}
 \Bigr)\abs{\cum_X(u_1, \dots, u_{k-1})} < \infty,
 \quad
 \text{for $k = 2, 3, \dots$}.
 \]
\end{assumption}
The asymptotic distribution of $\hat{\lambda}_p$ is given as follows.

\begin{theorem} \label{thm:3.3}
Suppose $\{X_t;\, t\in \Z\}$ satisfies Assumption \ref{asp:3.2} and
the $p$th quantile $\lambda_p$ of the spectral distribution
of $\{X_t\}$ is defined by \eqref{eq:2.2}. 
If 
$\hat{\lambda}_p$ is defined by \eqref{eq:2.7}, then we have
\begin{equation*}
\sqrt{n} \, \,(\hat{\lambda}_p - \lambda_p ) \dlim  \mathscr{E}^{-2} \mathcal{N}(0, \sigma^2),  
\end{equation*}
where $\mathscr{E}$ is a random variable distributed as exponential distribution
with mean $f_X(\lambda_p)$ and 
\begin{multline*}
  \sigma^2
 =
 \pi p^2 \int_{-\pi}^{\pi} f_X(\omega)^2 d\omega
 + 
 2\pi(1-4p) \int_{-\pi}^{\lambda_p} f_X(\omega)^2 d\omega \\
 +
 2\pi
 \Bigl\{
 \int_{-\pi}^{\lambda_p} \int_{-\pi}^{\lambda_p}
 Q_X(\omega_1, \omega_2, -\omega_2) d\omega_1 d\omega_2\\
 +
 \freqint \freqint p^2
 Q_X(\omega_1, \omega_2, -\omega_2) d\omega_1 d\omega_2\\
 -
 2p   \int_{-\pi}^{\lambda_p} \freqint
  Q_X(\omega_1, \omega_2, -\omega_2) d\omega_1 d\omega_2
 \Bigr\}.
\end{multline*}
The random variables $\mathscr{E}$ and $\mathcal{N}$ are correlated according to a quantity concerning
with the third order cumulants of the process $\{X_t\}$.
If the process $\{X_t\}$ is Gaussian or symmetric around 0, then $\mathscr{E}$ and $\mathcal{N}$
are independent.
\end{theorem}
Although the estimator $\hat{\lambda}_p$, defined by \eqref{eq:2.7}, is consistent,
the asymptotic distribution of $\hat{\lambda}_p$ is very hard to use in practice.
A modified estimator $\hat{\lambda}^*_p$ will given in the next section for quantile tests.

\section{Hypotheses testing for sinusoid models} \label{sec:4}
In this section, we consider the following testing problem ($\star$),
\begin{align}
 H:& \ Y_t = X_t \notag \\[0.25cm]
 &\text{versus}   \notag\\
 A: Y_t = \sum_{j = 1}^{J} &R_j \cos(\lambda_j t + \phi_j) + X_t, \label{eq:4.8}
\end{align}
where $\{X_t\}$ is a zero mean second order stationary process with finite autocovariance function $R_X(h)$
as before.
$\{\phi_j\}$ is uniformly distributed on $(-\pi, \pi)$, independent of $\{X_t\}$.
$\{R_j\}$ and $\{\lambda_j\}$ are real constants. 
In addition, suppose there exists at least one $R_j$ such that  $R_j \not = 0$.
In the alternative, the autocovariance function $R_Y(h)$ of $\{Y_t\}$ is 
\[
R_Y(h) = 
\frac{1}{2} \sum_{j = 1}^J R_j^2 \cos(\lambda_j h) + R_X(h).
\]
From \eqref{eq:2.1}, the spectral distribution function $F_Y(\omega)$ is represented by
\[
F_Y(\omega)
=
\frac{1}{2} \sum_{j = 1}^J R_j^2 \mathcal{H}(\omega - \lambda_j) + F_X(\omega),
\]
where $\mathcal{H}(\omega)$ is so called Heaviside step function such that
\[
\mathcal{H}(\omega)
=
\begin{cases}
1, \quad \text{if $\omega \geq 0$}, \\
0, \quad \text{otherwise}.
\end{cases}
\]
As for the alternative hypothesis, $F_Y(\omega) \not = F_X(\omega)$ if $\omega \not = -\pi$, $0$ or $\pi$.

As what we have seen in Section \ref{sec:3},
the asymptotic distribution of the estimator $\hat{\lambda}_p$ is peculiar with stronger assumptions
while it acts like a sandwich form.
We will modify $\hat{\lambda}_p$ by the method of smoothing.
We introduce the modified quantile estimator $\hat{\lambda}_p^*$
for the spectral distribution function of the sinusoid models $\{Y_t\}$
and test the null hypothesis $H$ by quantile test below.

Let us first introduce an extension of periodogram \eqref{eq:2.6} by
\[
I_{n, Y}^*(\omega)
=
\sum_{\abs{h} < n}
C_n^Y(h) \exp(-i h \omega),
\]
where $C_n^Y(h)$ is the sample autocovariance of $\{Y_t\}$.
The smoothed periodogram is defined based on a window function $A(\omega)$
such that
\begin{equation} \label{eq:4.9}
\hat{f}_Y(\omega)
=
\frac{1}{2\pi}
\sum_{\abs{h} \leq m}
\phi
\Bigl(
\frac{h}{m}
\Bigr) 
\freqint
I_{n, Y}^*(\lambda) \exp(-i h(\omega - \lambda)) d\lambda.
\end{equation}
Assumptions on the window function $\phi(\omega)$ are given as follows.
\begin{assumption} \label{asp:4.3}
 Let $\phi(\omega)$ satisfy
\begin{enumerate}[{\rm (i)}]
\item $m \to \infty$ and $m/n \to 0$ as $n \to \infty$.
\item $\phi(0) = 1$.
\item $\phi(-\omega) = \phi(\omega)$ and $\abs{\phi(\omega)} \leq 1$ for all $\omega \in \Lambda$.
\item $\phi(\omega) = 0$ for $\abs{\omega} > 1$.
\item The pair $(\phi, f_Y)$ satisfies $\phi(\cdot) f_Y(\cdot) \in \mathcal{L}^u$
for some $u$, $1 < u \leq 2$, and suppose that there exists $c > 0$ such that
\[
\sup_{\abs{\lambda} < \ep} \norm{\phi(\cdot) \{f_Y(\cdot) - f_Y(\cdot - \lambda)\}}_u = O(\ep^c)
\]
as $\ep \to 0$.
\end{enumerate}
\end{assumption}

Let us introduce the modified quantile estimator $\hat{\lambda}^*_p$.
Following \eqref{eq:2.6}, define the objective function $S^*_n(\theta)$ by
\[
S^*_n(\theta)
=
\freqint \rho_p(\omega - \theta) \hat{f}_Y(\omega) d\omega.
\]
The modified estimator $\hat{\lambda}^*_p$, then, is 
\begin{equation} \label{eq:4.10}
\hat{\lambda}^*_p = \arg \min_{\theta \in \Lambda} S^*_n(\theta).
\end{equation}
\begin{theorem} \label{thm:4.4}
Suppose $\{Y_t;\, t\in \Z\}$ is defined by \eqref{eq:4.8}.
The $p$th quantile $\lambda_p$ of the spectral distribution
of $\{Y_t\}$ is defined by \eqref{eq:2.2}. 
If 
$\hat{\lambda}_p^*$ is defined by \eqref{eq:4.10}, then we have
\[
\hat{\lambda}_p^* \plim \lambda_p. 
\]
\end{theorem}
The consistency of the modified estimator \eqref{eq:4.10}
do not require the continuity of the spectral distribution function $F_Y(\omega)$,
which can be considered as a stronger result than Theorem \ref{thm:3.2}.
We can use the modified estimator $\hat{\lambda}_p^*$ in practice
as a method to test the hypothesis of sinusoid models since 
$F_Y(\omega)$ is uniquely determined by its autocovariance function $R_Y(h)$.

Let us introduce quantile tests in frequency domain for sinusoid models.
The hypothesis testing problem ($\star$) can be changed into a general testing problem
\begin{align*}
 H:& \ \hat{\lambda}_p^* = \lambda_p \\[0.25cm]
 &\text{versus}  \\[0.25cm]
 A:& \ \hat{\lambda}_p^* \not= \lambda_p.
\end{align*}
Here, we consider the asymptotic distribution of the estimator $\hat{\lambda}_p^*$.

\begin{assumption} \label{asp:4.1}
The spectral distribution function $F_Y(\omega)$ has a density $f_Y(\omega)$ in a neighborhood of $\lambda_p$
and $f_Y(\omega)$ is continuous at $\lambda_p$ with $0 < f_Y(\lambda_p) < \infty$.
\end{assumption}
This assumption is not so strong since the jump points in the distribution are countable at most.
It is possible to choose a proper quantile or multiple quantiles as our interest to implement the hypothesis testing.

The asymptotic distribution of the modified estimator $\hat{\lambda}^*_p$ is given below.
\begin{theorem} \label{thm:4.5}
Suppose $\{Y_t;\, t\in \Z\}$ is defined by \eqref{eq:4.8}.
The $p$th quantile $\lambda_p$ of the spectral distribution
of $\{Y_t\}$ is defined by \eqref{eq:2.2}. 
If $\hat{\lambda}_p^*$ is defined by \eqref{eq:4.10}, 
then we have
\begin{equation} \label{eq:4.11}
\sqrt{n} (\hat{\lambda}_p^* - \lambda_p ) \to_d \mathcal{N}(0, \sigma^2),  
\end{equation}
where
\begin{multline*}
 \sigma^2
 =
 f(\lambda_p)^{-2}
 \Bigl[
 \pi p^2 \int_{-\pi}^{\pi} \phi(\omega)^2 f_Y(\omega) f_X(\omega) d\omega\\
 + 
 2\pi(1-4p) \int_{-\pi}^{\lambda} \phi(\omega)^2 f_Y(\omega) f_X(\omega) d\omega \\
 +
 2\pi
 \Bigl \{
 \int_{-\pi}^{\lambda_p} \int_{-\pi}^{\lambda_p}
 \phi(\omega)^2 Q_X(\omega_1, \omega_2, -\omega_2) d\omega_1 d\omega_2\\
 +
 \freqint \freqint p^2
 \phi(\omega)^2 Q_X(\omega_1, \omega_2, -\omega_2) d\omega_1 d\omega_2\\
 -
 2p   \int_{-\pi}^{\lambda_p} \freqint
 \phi(\omega)^2 Q_X(\omega_1, \omega_2, -\omega_2) d\omega_1 d\omega_2
 \Bigr\}
 \Bigr].
\end{multline*}
\end{theorem}
Theorem \ref{thm:4.5} holds for sinusoid models so it also can be applied to the null hypothesis.

Let us introduce the testing procedure for the quantile problem above.
From Theorem \ref{thm:4.5}, we have the following result.
Let $\mu_p$ be the $p$th quantile of the spectral distribution of $\{Y_t\}$ in the alternative hypothesis.

\begin{corollary} \label{cor:4.5}
Suppose $\{Y_t;\, t\in \Z\}$ is defined by \eqref{eq:4.8}.
The $p$th quantile $\lambda_p$ of the spectral distribution
of $\{Y_t\}$ is defined by \eqref{eq:2.2}
and $\hat{\lambda}_p^*$ is defined by \eqref{eq:4.10}. From \eqref{eq:4.11},
\begin{enumerate}[{\rm (i)}]
 \item Under the null hypothesis $H$, $\sqrt{n} (\hat{\lambda}_p^* - \lambda_p )/\sigma \dlim \mathcal{N}(0, 1)$;
 \item Under the alternative hypothesis $A$, $\sqrt{n} (\hat{\lambda}_p^* - \lambda_p )/\sigma  
 \dlim \mathcal{N}(\mu_p - \lambda_p, 1)$.
\end{enumerate}
\end{corollary}
The hypothesis is rejected if $\sqrt{n}\abs{\hat{\lambda}_p^* - \lambda_p}/\sigma > \Phi_{1 - \alpha/2}$,
where $\Phi_{1 - \alpha/2}$ is the $1 - \alpha/2$ percentage point of a standard normal distribution.

\section{Numerical Studies} \label{sec:5}
In this section, we implement the numerical studies to confirm the theoretical results in
Sections \ref{sec:3} and \ref{sec:4}.

\subsection{Numerical results for estimator $\hat{\lambda}_p$} \label{subsec:5.1}
First, we focus on the consistency of the estimator $\hat{\lambda}_p$
defined by \eqref{eq:2.7}.
Second order stationary processes considered here are
Gaussian white noise model, Gaussian MA(1) process with coefficient 0.9,
Gaussian AR(1) model with coefficient 0.9 and  Gaussian AR(1) model with coefficient -0.9.
The spectral distribution functions for these four models are given in Figure \ref{fig:1}.
The dependence structures of them are obviously different.

We estimated the quantile $\lambda_p$ of the spectral distribution function
by 30 samples, generated from each Gaussian stationary process.
The numerical results of the estimator $\hat{\lambda}_p$ only for $0.5 \leq p \leq 1$
are listed in Table \ref{tbl:1},
since the spectral distribution functions of real-valued stationary processes are symmetric.

\begin{table}[H]
\caption{the estimated quantiles $\hat{\lambda}_p$ of the spectral distribution with 30 samples}
\begin{center}
\begin{tabular}{|c||c|c|c|c|}
\hline
$p$ & White noise & MA(1) & AR(1) with 0.9 & AR(1) with -0.9 \\\hline
0.5 & 0.000 & 0.000 & 0.000 & 0.000 \\
0.6 & 0.305 & 0.211 & 0.026 & 2.940\\
0.7 & 1.187 & 0.576 & 0.055 & 3.030\\
0.8 & 1.564 & 0.891 & 0.092 & 3.074\\
0.9 & 2.093 & 1.235 & 0.190 & 3.109\\
1.0 & 3.142 & 3.142 & 3.142 & 3.142 \\ \hline
\end{tabular}
\end{center}
\label{tbl:1}
\end{table}%
We can see that the results in Table \ref{tbl:1} correspond to Figure \ref{fig:1} in Section \ref{sec:2}.
That is to say, the quantile of the spectral distribution function reflects the traits
of stationary processes.
Furthermore, we can make use of $\hat{\lambda}_p$ to seize the traits.

In general asymptotic theory,
if the estimator is asymptotically normal, 
then the estimates will be improved when the sample sizes get large.
However, as what we have shown in Section \ref{sec:3},
the estimator $\hat{\lambda}_p$ based on the bare periodogram is not asymptotically normal.
We next give the results in the white noise case with different sample size
to see the phenomenon.
The sample sizes are set to be 30, 50, 100 and 200.

\begin{table}[H]
\caption{the estimated quantiles $\hat{\lambda}_p$ in white noise case with different numbers of samples}
\begin{center}
\begin{tabular}{|c||c|c|c|c|}
\hline
$p \verb+\+\ n $ & 30 & 50 & 100 & 200 \\\hline
0.5 & 0.000 & 0.000 & 0.000 & 0.000 \\
0.6 & 0.305 & 0.663 & 0.366 & 0.745\\
0.7 & 1.187 & 1.226 & 0.966 & 1.260\\
0.8 & 1.564 & 1.990 & 1.602 & 1.881\\
0.9 & 2.093 & 2.440 & 2.251 & 2.334\\
1.0 & 3.142 & 3.142 & 3.142 & 3.142 \\ \hline
\end{tabular}
\end{center}
\label{tbl:2}
\end{table}%

From Table \ref{tbl:2}, 
we can see the accuracy is not quite improved when the sample size gets large.
This numerical result supports the theoretical results given in Theorem \ref{thm:3.3} 
in Section \ref{sec:3}, since, not only a normal distribution inside the asymptotic distribution
of $\hat{\lambda}_p$, the asymptotic distribution is also influenced by 
exponential distributed random variable.

At last, we would like to look at the behavior of the estimator $\hat{\lambda}_p$ for sinusoid models.
In addition to the same settings of $X_t$ given above, 
we add a harmonic component $m_t$ in the model with $\omega_0 = \pi/2$, i.e.
\begin{equation} \label{eq:5.12}
Y_t = m_t + X_t,
\end{equation}
where $m_t$ is defined in the following way: with uniformly distributed $\phi$ on $[-\pi, \pi]$
\[
m_t = 1/2 \cos (\omega_0 \, t + \phi).
\]

As already known, the spectral distribution function of $\{Y_t\}$ 
has a large change at the certain frequency 
$\omega_0 = \pi/2$.
Still, we estimated the quantile $\lambda_p$ by 30 samples, generated 
from the sinusoid models \eqref{eq:5.12}.
Compared with the results in Table \ref{tbl:1},
we can see that the estimated quintiles are pulled around to the frequency $\omega_0$
from Table \ref{tbl:3}.
Accordingly, even in the sinusoid models, 
the quantile $\lambda_p$ shows the phase of the spectral distribution function.
We can grasp them from the quantile estimator $\hat{\lambda}_p$.

\begin{table}[H]
\caption{the estimated quantiles $\hat{\lambda}_p$ of the spectral distribution with 30 samples}
\begin{center}
\begin{tabular}{|c||c|c|c|c|}
\hline
$p$ & White noise & MA(1) & AR(1) with 0.9 & AR(1) with -0.9 \\\hline
0.5 & 0.000 & 0.000 & 0.000 & 0.000 \\
0.6 & 1.399 & 0.412 & 0.030 & 2.610\\
0.7 & 1.513 & 0.789 & 0.065 & 3.014\\
0.8 & 1.577 & 1.254 & 0.116 & 3.066\\
0.9 & 1.679 & 1.582 & 1.152 & 3.106\\
1.0 & 3.142 & 3.142 & 3.142 & 3.142 \\ \hline
\end{tabular}
\end{center}
\label{tbl:3}
\end{table}%


\subsection{Statistical power of quantile tests in frequency domain}
Next, we implement quantile tests in frequency domain to see the performance of our testing 
procedure.
The Bartlett window is used for our purpose to smooth the periodogram \eqref{eq:2.6}.
To know the quantile $\lambda_p$ for each model is very difficult,
so we fixed $p = 0.7$ and $p = 0.8$ and numerically calculated $\lambda_p$ in advance.

Also, the theoretical result of the asymptotic variance $\sigma^2$ is also difficult to calculate.
We used the unbiased variance $\hat{\sigma}$ of the estimator in 100 simulations.
The significant level $\alpha$ is set to be $0.1$.

We set $\lambda_p$ as the true quantile for the null hypothesis.
Under the alternative models (Gaussian white noise model, 
Gaussian MA(1) model,
Gaussian AR(1) models as before),
50 samples are generated to estimate the quantile by the estimator $\hat{\lambda}_p^*$.
\begin{table}[H]
\caption{Statistical power of quantile tests for $\lambda_{0.7}$ with 50 samples}
\begin{center}
\begin{tabular}{|c||c|c|c|c|}
\hline
H $\backslash$ A & White noise & MA(1) & AR(1) with 0.9 & AR(1) with -0.9 \\\hline
White noise & -- & 0.99 & 1.00 & 1.00 \\ \hline
MA(1) & 1.00 & -- & 0.99 & 1.00\\ \hline
AR(1) with 0.9 & 1.00 & 1.00 & -- & 1.00\\ \hline
AR(1) with -0.9 & 1.00 & 1.00 &  1.00 & --\\ \hline
\end{tabular}
\end{center}
\label{tbl:4}
\end{table}%

\begin{table}[H]
\caption{Statistical power of quantile tests for $\lambda_{0.8}$ with 50 samples}
\begin{center}
\begin{tabular}{|c||c|c|c|c|}
\hline
H $\backslash$ A & White noise & MA(1) & AR(1) with 0.9 & AR(1) with -0.9 \\\hline
White noise & -- & 1.00 & 1.00 & 1.00 \\ \hline
MA(1) & 1.00 & -- & 0.99 & 1.00\\ \hline
AR(1) with 0.9 & 1.00 & 1.00 & -- & 1.00\\ \hline
AR(1) with -0.9 & 1.00 & 1.00 &  1.00 & --\\ \hline
\end{tabular}
\end{center}
\label{tbl:5}
\end{table}%

As what we can see from both Tables \ref{tbl:4} and \ref{tbl:5}, the statistical power is much high.
One reason to explain this result is that the dependence structures of these four models are quite different.
When $p$ is closer to $0.5$ or $1$, or the dependence structures of models are more similar, 
then the statistical power will be lower.

\section{Proofs of Theorems}
In this section, we provide proofs of theorems in the previous sections.
\begin{proof}[Theorem \ref{thm:2.1}]
First, we confirm the existence of the minimizer of $S(\theta)$.
The right derivative of $S(\theta)$ is
\[
S_+'(\theta) \equiv \lim_{\ep \to +0} \frac{S(\theta + \ep) - S(\theta)}{\ep} = F_Y(\theta) - p \Sigma_Y.
\]
From \eqref{eq:2.2}, we have 
\[
S_+'(\theta) 
\begin{cases}
 < 0, & \text{for $\theta < \lambda_p$},\\
 \geq 0, & \text{for $\theta \geq \lambda_p$}.
\end{cases}
\]
Thus, the minimizer of $S(\theta)$ exists and $S(\lambda_p) = \min_{\theta \in \Lambda} S(\theta)$.
The uniqueness of $\lambda_p$ and
the representation \eqref{eq:2.4} follow \eqref{eq:2.2}. 
\end{proof}

\begin{proof}[Theorem \ref{thm:3.2}]
Let $m$ be the minimum of $S(\theta)$.
The convexity of $S_n(\theta)$ is shown by the positiveness of 
the second derivative of $S_n(\theta)$, i.e.,
\begin{equation*}
 \fp{\theta}{2} S_n(\theta) = I_{n, X}(\theta) > 0 \quad \text{a.s.}
\end{equation*}
Now, let us consider the pointwise limit of $S_n(\theta)$. Actually, for each $\theta \in \Lambda$,
\begin{multline*}
\abs{S_n(\theta) - S(\theta)}
\leq
\Babs{\freqint \rho_p(\omega - \theta) (I_{n, X}(\omega) - E I_{n, X}(\omega)) d\omega} \\
+
\Babs{\freqint \rho_p(\omega - \theta) (E I_{n, X}(\omega) - f_X(\omega)) d\omega}.
\end{multline*}
The first term in right hand side converges to 0 in probability,
which can be shown by the summability of the fourth order cumulants under Assumption \ref{asp:3.1} (i).
The second term in right hand side converges to 0 under Assumption \ref{asp:3.1} (ii).
(See \cite{hannan1970, ht1982}).
By the Convexity Lemma in \cite{pollard1991},
\begin{equation} \label{eq:6.13}
 \sup_{\theta \in K} \abs{S_n (\theta) - S(\theta)} \plim 0,
\end{equation}
for any compact subset $K \subset \Lambda$. 

Let $B(\lambda_p)$ be any open neighborhood of $\lambda_p$.
From the uniqueness of zero of $S(\theta)$,
there exists an $\epsilon > 0$ 
such that $\inf_{\mu \in \Lambda/B(\lambda) } \abs{S(\mu)} > m + \epsilon$.
Thus, with probability tending to 1,
\[
\inf_{\theta \in \Lambda/B(\lambda_p) } S_n(\theta)
\geq
\inf_{\theta \in \Lambda/B(\lambda_p) } S(\theta)
-
\sup_{\theta \in \Lambda/B(\lambda_p)} \abs{S(\theta) - S_n(\theta)} > m,
\]
where it is implied by \eqref{eq:6.13} that the second term can be chosen arbitrarily small.
The conclusion follows that with probability tending to 1, $S_n(\hat{\lambda}_p) \leq m - \epsilon^*$ 
by the pointwise convergence of $S_n(\theta)$ in probability.

\end{proof}

To prove Theorem \ref{thm:3.3}, we first consider asymptotic variance of
\begin{equation} \label{eq:6.14}
  T_n(\lambda) \equiv n^{\beta} \int_{\lambda}^{\lambda + n^{-\beta}} I_{n, X}(\omega) d\omega.
\end{equation}
The asymptotic variance can be classified as the following lemma.
\begin{lemma} \label{lem:6.2}
Suppose $\{X(t)\}$ satisfies Assumption \ref{asp:3.2}. Let $T_n(\lambda)$ be defined as \eqref{eq:6.14}.
Then the asymptotic variance of $T_n(\lambda)$ is given by
  \[
  \lim_{n \to \infty}\var(T_n(\lambda))
  = 
\begin{cases}
 0, & \quad \text{if $\beta < 1$},\\
 f_X(\lambda)^2, & \quad \text{if $\beta = 1$}, \\
 \infty, & \quad \text{if $\beta > 1$}.
\end{cases}
  \]
\end{lemma}
\begin{proof}
Let $a_n = n^{\beta}$.
Divide $T_n(\lambda)$ by
\[
a_n \int_{-\pi}^{\lambda + a_n^{-1}} I_{n, X}(\omega) d\omega
-
a_n \int_{-\pi}^{\lambda} I_{n, X}(\omega) d\omega.
\]
The variances of both two parts and their covariance are given by 
\begin{multline*}
 \var\Bigl( a_n \int_{-\pi}^{\lambda + a_n^{-1}} I_{n, X}(\omega) d\omega\Bigr)
 =\\
\frac{a_n^2}{n}\,
 2\pi
\Bigl(
\int_{-\pi}^{\lambda+a_n^{-1}} f_X(\omega)^2 d\omega
+
\int_{-\pi}^{\lambda+a_n^{-1}}
\int_{-\pi}^{\lambda+a_n^{-1}} 
Q_X(\omega_1, \omega_2, -\omega_2) d\omega_1 d\omega_2
\Big),
\end{multline*}
\begin{multline*}
 \var\Bigl( a_n \int_{-\pi}^{\lambda} I_{n, X}(\omega) d\omega\Bigr)
 =\\
\frac{a_n^2}{n} \,
 2\pi
\Bigl(
\int_{-\pi}^{\lambda} f_X(\omega)^2 d\omega
+
\int_{-\pi}^{\lambda}
\int_{-\pi}^{\lambda} 
Q_X(\omega_1, \omega_2, -\omega_2) d\omega_1 d\omega_2
\Big),
\end{multline*}
and
\begin{multline*}
 \cov\Bigl(a_n \int_{-\pi}^{\lambda + a_n^{-1}} I_{n, X}(\omega) d\omega,
 a_n \int_{-\pi}^{\lambda} I_{n, X}(\omega) d\omega\Bigr)\\
 =
\frac{a_n^2}{n}\,
 2\pi
\Bigl(
\int_{-\pi}^{\lambda} f_X(\omega)^2 d\omega
+
\int_{-\pi}^{\lambda}
\int_{-\pi}^{\lambda+a_n^{-1}} 
Q_X(\omega_1, \omega_2, -\omega_2) d\omega_1 d\omega_2
\Big).
\end{multline*}
As a result, the variance of $T_n(\lambda)$ is
\begin{multline} \label{eq:6.15}
\var(T_n(\lambda))
=
 \frac{a_n^2}{n}\,
 2\pi
\Bigl(
\int_{\lambda}^{\lambda+a_n^{-1}} f_X(\omega)^2 d\omega\\
+
\int_{\lambda}^{\lambda+a_n^{-1}}
\int_{-\pi}^{\lambda+a_n^{-1}} 
Q_X(\omega_1, \omega_2, -\omega_2) d\omega_1 d\omega_2\\
-
\int_{-\pi}^{\lambda}
\int_{\lambda}^{\lambda+a_n^{-1}} 
Q_X(\omega_1, \omega_2, -\omega_2) d\omega_1 d\omega_2
\Big).
\end{multline}
We can see the result from \eqref{eq:6.15} by cases:
\begin{enumerate}[(i)]
 \item if $a_n = n^{\beta}$ where $0 < \beta < 1$, then the limiting variance of $T_n(\lambda)$ is
\[
\var(T_n(\lambda))
\to
0,
\]
\item if $a_n = n^{\beta}$ where $\beta > 1$, then the limiting variance of $T_n(\lambda)$ is
\[
\var(T_n(\lambda)) \to \infty,
\]
\item if $a_n = n^{\beta}$ where $\beta = 1$, then the limiting variance of $T_n(\lambda)$ is
\[
\var(T_n(\lambda)) \to  f_X(\lambda)^2.
\]
\end{enumerate}
Thus, the conclusion holds.

\end{proof}

\begin{remark}
The result in Lemma \ref{lem:6.2} seems surprising at first glance,
since it may be expected that \eqref{eq:6.14} do not depend on the order of factor $n^{\beta}$.
However, the phenomenon can be explained in a heuristic way. 
Returning back to the definition of $T_n(\lambda)$,
the quantity
\[
\int_{\lambda}^{\lambda + n^{-\beta}} I_{n, X}(\omega) d\omega
\]
is approximated by the following discrete statistic
\begin{equation} \label{eq:6.16}
\frac{2\pi}{n} 
\sum_{\lambda \leq 2\pi s/n \leq \lambda + n^{-\beta}} 
I_{n, X}
\Bigl(
\frac{2\pi s}{n}
\Bigr).
\end{equation}
Looking at the number of periodograms $I_{n, X}(\lambda_s)$ with different frequencies,
we can find that \eqref{eq:6.16} depends on the order of $n^{-\beta}$.
If $0 < \beta < 1$, then more and more periodograms will be involved in the summation
as $n$ increases.
Conversely, if $\beta > 1$,
then the interval for the frequency will be much smaller as $n$ increases.
Only the case $\beta = 1$ keeps the same order between the number of periodograms
and the length of the interval, and therefore only one periodogram
$I_{n, X}(2\pi s/n)$ is involved in the summation.
\end{remark}

Next, we have to consider the domain of periodogram on the lattice as in
\cite{brillinger2001}.
That is to say,
for any $\omega \in [-\pi, \pi]$, define periodogram $I_{n, X}(\omega)$ discretely by $I_{n, X}(\omega_k)$,
where $\omega_k$ is defined as the closest frequency of the multiple of $2\pi/n$.
It is easy to see that
\[
\abs{I_{n, X}(\omega) - I_{n, X}(\omega_k)} = o_p(1).
\]

\begin{lemma} \label{lem:6.3}
If $\omega_k \not= -\pi, 0, \pi$, then the random vector 
 \[
 \sqrt{n} \,\, \Bigl(
 \frac{1}{n}\sum_{t=1}^n X_t \cos(\omega_k t), 
 \frac{1}{n}\sum_{t=1}^n X_t \sin(\omega_k t) \Bigr)'
 \]
 has a joint asymptotic normal distribution with the covariance matrix $1/2\,\Sigma_X I_2$.
\end{lemma}

\begin{proof}
 Obvious. 
\end{proof}

Then, let $C_n(m)$ be the sample autocovariance, i.e.
\[
C_n(m) = \frac{1}{n - m} \sum_{s=1}^{n-m} X_s X_{s+m}.
\]
The joint distribution of the random vector
$\sqrt{n} \,(C_n(1) - R_X(1), \cdots, C_n(l) - R_X(l), \\ 
(1/n) \sum_{t=1}^n X_t \cos(\omega_k t), (1/n)\sum_{t=1}^n X_t \sin(\omega_k t))'$
will be considered in the next lemma.
The result is applied to show the asymptotic distribution of $\sqrt{n} (\hat{\lambda}_p - \lambda_p)$.
\begin{lemma} \label{lem:6.4}
 Under Assumptions \ref{asp:3.2},
 the asymptotic joint distribution of the sample autocovariances
 and the trigonometric transforms ($\omega_k \not = -\pi, 0, \pi$) of samples
 is given by
\begin{equation} \label{eq:6.17}
\sqrt{n}\,\,
\begin{pmatrix}
 C_n(1) - R_X(1)\\
 \vdots\\
 C_n(l) - R_X(l)\\[0.2cm]
\frac{1}{n} \sum_{t=1}^n X_t \cos(\omega_k t)\\[0.2cm]
\frac{1}{n} \sum_{t=1}^n X_t \sin(\omega_k t)
 \end{pmatrix}
 \dlim
 \mathcal{N}(\bm{0},  
\begin{pmatrix}
V & \Delta_3 & \Delta_3\\
\Delta_3' & \frac{1}{2}\Sigma_X & 0\\
\Delta_3' & 0 & \frac{1}{2}\Sigma_X
\end{pmatrix}
),
\end{equation}
where the matrix $V$ is given by
\begin{multline*}
V_{m_1 m_2} =
 2\pi \freqint f_X^2(\omega) \{\exp(-i (m_2 - m_1) \omega) +  \exp(i (m_2 + m_1) \omega)\}d\omega\\
 +
 (2\pi)^{-2} \freqint\freqint
 \exp(i m_1 \omega_1 + i m_2 \omega_2)
 Q_X(\omega_1, - \omega_2, \omega_2) d\omega_1 d\omega_2.
 \end{multline*}
The $l$-vector $\Delta_3$ is a quantity defined in the proof, which is related to the third order cumulants
of the stochastic process $\{X_t\}$.
\end{lemma}

\begin{proof}
The statement will be shown by Cram{\'e}r-Wold device.
Suppose $\bm{q} = (q_1, \dots, \\q_{l+2})$ and 
$(X_{n+1}, \dots, X_{n + l})$ is generated from the stationary process $\{X_t;\, t \in \Z\}$.
Then, we can define a random vector $\tilde{S}_t$ as
\[
\tilde{S}_t
=
(X_t X_{t + 1} - R_X(1), \cdots, X_t X_{t + l} - R_X(l), 
 X_t \cos(\omega_k t), X_t \sin(\omega_k t))'.
\]
Denote the left hand side of \eqref{eq:6.17} by $S_n$.
It is not difficult to see that
\[
\Babs{\frac{1}{n} \sumn[t] \tilde{S}_t - S_n}
\plim
0,
\]
since $(X_{n+1}, \dots, X_{n + l})$ is bounded.
Let us consider the random variable $\bm{q}' \tilde{S}_t$.
It holds that $E (\bm{q}' \tilde{S}_t )= 0$.
Denote the variance of $\bm{q}' \tilde{S}_t$ by $s_n = \var(\bm{q}' \tilde{S}_t)$.
Under Assumption \ref{asp:3.2}, we can find that, from \cite{ht1982},
\[
\cov(C_n(i), C_n(j)) = 
O
\Bigl(
\frac{1}{n}
\Bigr),
\]
for $i, j = 1, 2, \dots, l$, from Lemma \ref{lem:6.3}, 
\[
\var 
\Bigl(
\frac{1}{n} \sum_{t=1}^n X_t \cos(\omega_k t)
\Bigr)
=
O
\Bigl(
\frac{1}{n}
\Bigr),
\quad
\var 
\Bigl(
\frac{1}{n} \sum_{t=1}^n X_t \sin(\omega_k t)
\Bigr)
=
O
\Bigl(
\frac{1}{n}
\Bigr),
\]
and for any $1 \leq m \leq l$,
\begin{multline} \label{eq:6.18}
 \cov
 \Bigl(
 C_n(m), n^{-1}\sum_{t=1}^n X_t \cos(\omega_k t)
 \Bigr)\\
 =
 n^{-1} (n-m)^{-1}\sum_{s=1}^{n-m} \sumn[t] \cos(\omega_k t) \cum(X_s, X_{s+m}, X_t).
\end{multline}
Under Assumption \ref{asp:3.2}, the right hand side of \eqref{eq:6.18}
can be bounded by
\begin{multline*}
n^{-1} (n-m)^{-1}\sum_{s=1}^{n-m} \sumn[t] \cos(\omega_k t) \cum(X_s, X_{s+m}, X_t)\\
\leq
\frac{1}{n} \sum_{k = 1 - n}^{n - 1}
\Bigl(
1 - \frac{\abs{k}}{n} 
\Bigr)
\cum_X(m, k)
=
O
\Bigl(
\frac{1}{n}
\Bigr).
\end{multline*}
Thus, for any $\ep > 0$,
\[
n^{-1} \sumn[t] E((\bm{j}' \tilde{S}_t)^2 \ind(\abs{\bm{j}' \tilde{S}_t} >n^{1/2}  \epsilon )) \to 0,
\]
as $n \to \infty$. Now if we define
\[
\Delta_3(m) \equiv
\lim_{n \to \infty} 
\frac{1}{n} \sum_{s=1}^{n} \sum_{t=1}^{n}
\cos(\omega_k t) \cum(X_s, X_{s+m}, X_t),
\]
then $\Delta_3 = (\Delta_3(1), \dots, \Delta_3(l))'$.
By Lindeberg's central limit theorem, $n^{-1/2} \sumn[t] \bm{q}' \tilde{S_t}$
is asymptotically Gaussian distributed. 
The conclusion follows Cram{\'e}r-Wold device.
\end{proof}

Following Lemma \ref{lem:6.2}, Lemma \ref{lem:6.3} and Lemma \ref{lem:6.4}, 
we give the proof of Theorem \ref{thm:3.3}.
\begin{proof}[Theorem \ref{thm:3.3}]
Consider the following process
\begin{equation*}
M_n(\delta) = n\
\Bigl\{S_n (\lambda_p - \frac{\delta}{\sqrt{n}}) - S_n(\lambda_p)
\Bigr\}.
\end{equation*}
By Knight's identity (see \cite{knight1998}), we have
\begin{eqnarray*}
 M_n(\delta) &=& 
 -\delta \sqrt{n}
 \lprt{\freqint (p - \ind(\omega < \lambda_p)) (I_{n, X}(\omega) - f_X(\omega))d\omega}  \\
 && \quad \quad \quad +
 \freqint \int_0^{\delta/\sqrt{n}} n \,  (\ind(\omega \leq \lambda_p + s) 
 -
\ind(\omega \leq \lambda_p) I_{n, X}(\omega)  ds d\omega \\
 &=&
 M_{n1}(\delta) + M_{n2}(\delta), \quad ({\rm say}).
\end{eqnarray*}
Under Assumption \ref{asp:3.2}, we have,
by Theorem 7.6.3 in \cite{brillinger2001}, 
\[
 M_{n1}(\delta)
 \dlim
 -\delta\mathcal{N}(0, \sigma^2),
\]
where
\begin{multline*}
 \sigma^2
 =
 \pi p^2 \int_{-\pi}^{\pi} f_X(\omega)^2 d\omega
 + 
 2\pi(1-4p) \int_{-\pi}^{\lambda_p} f_X(\omega)^2 d\omega \\
 +
 2\pi
 \Bigl\{
 \int_{-\pi}^{\lambda_p} \int_{-\pi}^{\lambda_p}
 Q_X(\omega_1, \omega_2, -\omega_2) d\omega_1 d\omega_2\\
 +
 \freqint \freqint p^2
 Q_X(\omega_1, \omega_2, -\omega_2) d\omega_1 d\omega_2\\
 -
 2p   \int_{-\pi}^{\lambda_p} \freqint
  Q_X(\omega_1, \omega_2, -\omega_2) d\omega_1 d\omega_2
 \Bigr\}.
\end{multline*}
From Lemma \ref{lem:6.2}, in view of
\begin{equation} \label{eq:6.19}
n \int_{\lambda_p}^{\lambda_p + n^{-1}} I_{n, X}(\omega) d\omega
\to
I_{n, X}(\lambda_p) \quad \text{a.s.},
\end{equation}
we will use \eqref{eq:6.19} to evaluate $M_{n2}(\delta)$.
The second term $M_{n2}(\delta)$ can be evaluated by
\begin{eqnarray*}
 M_{n2}(\delta)
 &=&
 \int_{0}^{\delta/\sqrt{n}} 
 \int_{\lambda_p}^{\lambda_p+s} n \, I_{n, X}(\omega) d\omega ds \\
 &=&
 \int_0^{\delta \sqrt{n}} \Bigl (
 n\int_{\lambda_p}^{\lambda_p + t/n} \, I_{n, X}(\omega) d\omega \Bigr) dt\\
 &=&
 \int_0^{\delta \sqrt{n}} t \, \,I_{n, X}(\lambda_p)\,dt\\
 &=&
 \frac{1}{2} \,\, I_{n, X}(\lambda_p) \,\, \delta^2 \quad \text{a.s.}
\end{eqnarray*}
This term, actually, does not converge in probability,
but has an asymptotic exponential distribution $\mathscr{E}$, which has mean
$f(\lambda)$.
Applying continuous mapping theorem to the result in Lemma \ref{lem:6.4},
the following joint distribution converges in distribution, i.e.,
\[
\begin{pmatrix}
 M_{n1}(\delta)\\[0.2cm]
 M_{n2}(\delta)
\end{pmatrix}
\dlim
\begin{pmatrix}
 \mathcal{N}\\[0.2cm]
 \mathscr{E}
\end{pmatrix}.
\]
Then by continuous mapping theorem again, we obtain
\[
M_n(\delta) \dlim M(\delta) = - \delta \mathcal{N} + \frac{1}{2} \delta \mathscr{E}^2,
\]
which is minimized by $\delta = \mathscr{E}^{-1} \mathcal{N}$.
In conclusion,
\[
\sqrt{n} (\hat{\lambda}_p - \lambda ) \dlim \mathscr{E}^{-2}\mathcal{N}(0, \sigma^2),
\]
From Lemma \ref{lem:6.4}, it can be seen that 
the dependence relationship between random variables
$\mathscr{E}$ and $\mathcal{N}$ depends on $\Delta_3$, i.e.,
the third cumulants of the process $\{X_t\}$.
If $\{X_t\}$ is Gaussian or symmetric around 0, then $\Delta_3 = 0$,
which implies that $\mathscr{E}$ and $\mathcal{N}$ are independent.

\end{proof}

Below, we provide the proof of Theorem \ref{thm:4.4}.
First, an extension of Lemma A2.2 in \cite{ht1982} is given in the following.

\begin{lemma} \label{lem:6.5}
Assume $\sum_{j_1, j_2, j_3 = -\infty}^{\infty} \abs{Q_X(j_1, j_2, j_3)} < \infty$.
For any square-integrable function $\phi(\omega)$,
\begin{equation} \label{eq:6.21}
 \freqint (I_{n, Y}(\omega) - EI_{n, Y}(\omega))\phi(\omega) d\omega \plim 0. 
\end{equation}
\end{lemma}

\begin{proof}
Let
 \[
 \tilde{\phi}(n) = \frac{1}{2\pi} \freqint \phi(\omega) \exp(i n \omega)d\omega.
 \]
From \cite{ht1982} and \cite{li1994}, it holds that
\begin{multline*}
\var \lprt{ \freqint (I_{n, Y}(\omega) - EI_{n, Y}(\omega))\phi(\omega) d\omega}
=\\
\frac{1}{n^2}
\sum_{t_1, t_2, t_3, t_4 = 1}^n
 \tilde{\phi}(t_1 - t_2)
  \tilde{\phi}(t_3 - t_4)
  \prt{
  R_Y(t_3 - t_1) R_{Y}(t_4 - t_2)\\
  +
  R_Y(t_4 - t_1) R_{Y}(t_3 - t_2)
  +
  Q_Y(t_2- t_1, t_3 - t_1, t_4 - t_1)}\\
  =
 \frac{2\pi}{n} \freqint (\phi(\omega) \overline{\phi(\omega)} + \phi(\omega) \overline{\phi(-\omega)} ) f_Y(\omega) f_X(\omega) d\omega\\
 +
 \frac{2\pi}{n} \freqint \freqint \phi(\omega_1) \phi(-\omega_2) Q_X(\omega_1, \omega_2, -\omega_2) d\omega_1 d\omega_2.
\end{multline*}
Here, $f_Y(\omega) = f_d(\omega) + f_X(\omega)$ if we suppose $f_d(\omega) = \frac{1}{4 \pi} R_j^2 \delta(\omega - \lambda_j)$,
where $\delta(\omega)$ is the Dirac delta function.
From Chebyshev's inequality, \eqref{eq:6.21} holds.
\end{proof}

\begin{proof}[Theorem \ref{thm:4.4}]
We only have to show the pointwise limit of $S_n^*(\theta)$ is given by $S(\theta)$.
The rest of argument for the proof follows the proof of Theorem \ref{thm:3.2}.
Note that $\hat{f}_Y(\omega)$ has a representation such that
\[
\hat{f}_Y (\omega) =
\freqint \phi(\omega - \lambda) I_{n, Y}^*(\lambda) d\lambda. 
\] 
Similarly, we have
\begin{multline*}
\abs{S_n^*(\theta) - S(\theta)}
\leq
\Babs{\freqint \rho_p(\omega - \theta) (\hat{f}_Y(\omega) - E \hat{f}_Y(\omega)) d\omega} \\
+
\Babs{\freqint \rho_p(\omega - \theta) E (\hat{f}_Y(\omega) ) d\omega
- 
\freqint \Bigl(
\freqint \rho_p(\omega - \theta) \phi(\omega - \lambda) d\omega
\Bigr)  
F_Y(d\lambda)}.
\end{multline*}
The first term in right hand side converges to 0 in probability,
which can be seen from Lemma \ref{lem:6.5}.
Under Assumption \ref{asp:4.3} (v), we see that
the second term in right hand side converges to 0 from Theorem 1.1 in \cite{hosoya1997}.

\end{proof}

Last, we give the proof of Theorem \ref{thm:4.5}.
\begin{proof}[Theorem \ref{thm:4.5}]
Consider the following process
\begin{equation*}
M_n^*(\delta) = n\
\Bigl\{S_n^* (\lambda_p - \frac{\delta}{\sqrt{n}}) - S_n^*(\lambda_p)
\Bigr\}.
\end{equation*}
By Knight's identity, we have
\begin{eqnarray*}
 M_n^*(\delta) &=& 
 -\delta \sqrt{n}
 \lprt{\freqint (p - \ind(\omega < \lambda_p)) \hat{f}_Y(\omega)d\omega}  \\
 && \quad \quad \quad +
 \freqint \int_0^{\delta/\sqrt{n}} n \,  (\ind(\omega \leq \lambda_p + s) 
 -
\ind(\omega \leq \lambda_p) \hat{f}_Y(\omega)  ds d\omega \\
 &=&
 M^*_{n1}(\delta) + M^*_{n2}(\delta), \quad ({\rm say}).
\end{eqnarray*}
From Lemma \ref{lem:6.5}, we can see that
\[
 M^*_{n1}(\delta)
 \dlim
 -\delta\mathcal{N}(0, \tilde{\sigma}^2),
\]
where $\sigma^2$ is
\begin{multline*}
 \tilde{\sigma}^2
 =
 \pi p^2 \int_{-\pi}^{\pi} \phi(\omega)^2 f_Y(\omega) f_X(\omega) d\omega
 + 
 2\pi(1-4p) \int_{-\pi}^{\lambda} \phi(\omega)^2 f_Y(\omega) f_X(\omega) d\omega \\
 +
 2\pi
 \Bigl\{
 \int_{-\pi}^{\lambda_p} \int_{-\pi}^{\lambda_p}
 Q_X(\omega_1, \omega_2, -\omega_2) d\omega_1 d\omega_2\\
 +
 \freqint \freqint p^2
 Q_X(\omega_1, \omega_2, -\omega_2) d\omega_1 d\omega_2\\
 -
 2p   \int_{-\pi}^{\lambda_p} \freqint
  Q_X(\omega_1, \omega_2, -\omega_2) d\omega_1 d\omega_2
 \Bigr\}.
\end{multline*}
As for the second term $M_{n2}(\delta)$, we have, under Assumptions \ref{asp:4.3} and \ref{asp:4.1},
\begin{eqnarray*}
 M^*_{n2}(\delta)
 &=&
 \int_{0}^{\delta/\sqrt{n}} 
 \int_{\lambda_p}^{\lambda_p+s} n \, \hat{f}_Y(\omega) d\omega ds \\
 &=&
 \int_0^{\delta \sqrt{n}} \Bigl (
 n\int_{\lambda_p}^{\lambda_p + t/n} \, \hat{f}_Y(\omega) d\omega \Bigr) dt\\
 &\plim&
 \frac{1}{2} \,\, f_Y(\lambda_p) \, \delta^2.
\end{eqnarray*}
Applying continuous mapping theorem to $M_n$, we obtain
\[
M_n^*(\delta) \dlim M^*(\delta) = - \delta \mathcal{N} +  \frac{1}{2} \,\, f_Y(\lambda_p) \, \delta^2,
\]
which is minimized by $\delta = f_Y(\lambda_p)^{-1} \mathcal{N}$.
Therefore,
\[
\sqrt{n} (\hat{\lambda}_p - \lambda ) \dlim \mathcal{N}(0, f_Y(\lambda_p)^{-2} \tilde{\sigma}^2),
\]
and the asymptotic variance $\sigma^2$ in Theorem \ref{thm:4.5} is 
$\sigma^2 = f_Y(\lambda_p)^{-2} \tilde{\sigma}^2$.

\end{proof}


\bibliographystyle{plain} 
\nocite{*}
\bibliography{freq_Liu}   

\begin{thebibliography}{10}

\bibitem{brillinger2001}
David~R Brillinger.
\newblock {\em Time Series: Data Analysis and Theory}, volume~36.
\newblock Siam, 2001.

\bibitem{dette2015}
Holger Dette, Marc Hallin, Tobias Kley, Stanislav Volgushev, et~al.
\newblock Of copulas, quantiles, ranks and spectra: An $ l\_ $\{$1$\}$
  $-approach to spectral analysis.
\newblock {\em Bernoulli}, 21(2):781--831, 2015.

\bibitem{hannan1973}
Edward~J Hannan.
\newblock The estimation of frequency.
\newblock {\em Journal of Applied probability}, 10:510--519, 1973.

\bibitem{hannan1970}
EJ~Hannan.
\newblock {\em Multiple Time Series}.
\newblock John Wiley \& Sons, 1970.

\bibitem{hosoya1989}
Yuzo Hosoya.
\newblock The bracketing condition for limit theorems on stationary linear
  processes.
\newblock {\em The Annals of Statistics}, 17(1):401--418, 1989.

\bibitem{hosoya1997}
Yuzo Hosoya.
\newblock A limit theory for long-range dependence and statistical inference on
  related models.
\newblock {\em The Annals of Statistics}, 25(1):105--137, 1997.

\bibitem{ht1982}
Yuzo Hosoya and Masanobu Taniguchi.
\newblock A central limit theorem for stationary processes and the parameter
  estimation of linear processes.
\newblock {\em The Annals of Statistics}, 10:132--153, 1982.

\bibitem{knight1998}
Keith Knight.
\newblock Limiting distributions for {\rm $l\sb{1}$} regression estimators
  under general conditions.
\newblock {\em The Annals of Statistics}, 26(2):755--770, 1998.

\bibitem{koenker2005}
Roger Koenker.
\newblock {\em Quantile Regression}.
\newblock Number~38. Cambridge university press, 2005.

\bibitem{li2008}
Ta-Hsin Li.
\newblock Laplace periodogram for time series analysis.
\newblock {\em Journal of the American Statistical Association},
  103(482):757--768, 2008.

\bibitem{li2012}
Ta-Hsin Li.
\newblock Quantile periodograms.
\newblock {\em Journal of the American Statistical Association},
  107(498):765--776, 2012.

\bibitem{li1994}
Ta-Hsin Li, Benjamin Kedem, and Sid Yakowitz.
\newblock Asymptotic normality of sample autocovariances with an application in
  frequency estimation.
\newblock {\em Stochastic Processes and their Applications}, 52(2):329--349,
  1994.

\bibitem{pollard1991}
David Pollard.
\newblock Asymptotics for least absolute deviation regression estimators.
\newblock {\em Econometric Theory}, 7(02):186--199, 1991.

\bibitem{qt1991}
B~G Quinn and P~J Thomson.
\newblock Estimating the frequency of a periodic function.
\newblock {\em Biometrika}, 78(1):65--74, 1991.

\bibitem{qh2001}
Barry~G Quinn and Edward~James Hannan.
\newblock {\em The estimation and tracking of frequency}, volume~9.
\newblock Cambridge University Press, 2001.

\bibitem{rr1988}
John~A Rice and Murray Rosenblatt.
\newblock On frequency estimation.
\newblock {\em Biometrika}, 75(3):477--484, 1988.

\bibitem{walker1971}
A~M Walker.
\newblock On the estimation of a harmonic component in a time series with
  stationary independent residuals.
\newblock {\em Biometrika}, 58(1):21--36, 1971.

\bibitem{whittle1952b}
Peter Whittle.
\newblock Tests of fit in time series.
\newblock {\em Biometrika}, 39(3):309--318, 1952.

\end{thebibliography}
\end{document}